\documentclass{article}
\usepackage{amsmath,amsthm,amsopn,amstext,amscd,amsfonts,amssymb}
\usepackage{natbib}

\def \tr {\mathop{\rm tr}\nolimits}
\def \re {\mathop{\rm Re}\nolimits}
\def \im {\mathop{\rm Im}\nolimits}

\def \Vol {\mathop{\rm Vol}\nolimits}

\def \etr {\mathop{\rm etr}\nolimits}
\def \diag {\mathop{\rm diag}\nolimits}

\renewenvironment{abstract}
                 {\vspace{6pt}
                  \begin{center}
                  \begin{minipage}{5in}
                  \centerline{\textbf{Abstract}}
                  \noindent\ignorespaces
                 }
                 {\end{minipage}\end{center}}
\newtheorem{thm}{\textbf{Theorem}}[section]
\newtheorem{cor}{\textbf{Corollary}}[section]
\newtheorem{lem}{\textbf{Lemma}}[section]
\theoremstyle{definition}

\title{\Large \textbf{On Wishart distribution}}
\author{
  \textbf{Jos\'e A. D\'{\i}az-Garc\'{\i}a} \thanks{Corresponding author\newline
    }\\
  {\normalsize Department of Statistics and Computation} \\
  {\normalsize 25350 Buenavista, Saltillo, Coahuila, Mexico} \\
  {\normalsize E-mail: jadiaz@uaaan.mx} \\[2ex]
  \textbf{Ram\'on Guti\'errez J\'aimez} \\
  {\normalsize Department of Statistics and O.R.} \\
  {\normalsize University of Granada} \\
  {\normalsize Granada 18071, Spain}\\
  {\normalsize E-mail: rgjaimez@ugr.es}\\
}
\date{}
\begin{document}
\maketitle

\begin{abstract}
This paper proposes a unified approach that enables the Wishart distribution to be studied
simultaneously in the real, complex, quaternion and octonion cases. In particular, the
noncentral generalised Wishart distribution, the joint density of the eigenvalues and the
distribution of the maximum eigenvalue are obtained for real normed division algebras.
\end{abstract}

\section{Introduction}\label{sec1}

Many results first described in statistical theory are then found in real cases, and the
version for complex cases is subsequently studied. This has been described in various papers,
see Bravais (1846)(cited by \citet{w:56}) and \citet{w:56}; \citet[Sections 4 and 8]{j:64};
\citet{m:82} and \citet{rva:05a}, among many other examples.

Using some concepts and results derived from abstract algebra, it is possible to propose a
unified means of addressing not only real and complex cases but also the quaternion and
octonion cases. Part of this approach has been used for some time in random matrix theory, see
\citet{er:05} and \citet{f:05}.

For the sake of completeness, in the present study the case of octonions is considered, but it
should be noted that many results for the octonion case can only be conjectured, because there
remain many unresolved theoretical problems in this respect, see \citet{dm:99}. Furthermore,
the relevance of the octonion case for understanding the real world has yet to be clarified,
see \citet{b:02}.

The rest of this paper is structured as follows: Section \ref{sec2} reviews some definitions
and notation on real normed division algebras. Some results for Jacobians and generalised
hypergeometric functions, together with an extension of one of the basic properties of zonal
polynomials (which is also valid for Jack polynomials) are also given. Section \ref{sec3} then
derives the noncentral generalised Wishart distribution, and as a corollary the noncentral
inverse generalised Wishart distribution is obtained. In Section \ref{sec4}, we obtain the
joint density function of the eigenvalues and the distribution of the maximum eigenvalue, the
latter under a matrix multivariate normal distribution. All these results are obtained for real
normed division algebras.

\section{Preliminary results}\label{sec2}

Let us introduce some notation and results that will be useful.

\subsection{Notation and real normed division algebras}\label{sec21}

A detailed discussion of real normed division algebras may be found in \citet{b:02}. For
convenience, we shall introduce some notations, although in general we adhere to standard
notations.

For our purposes, a \textbf{vector space} is always a finite-dimensional module over the field
of real numbers. An \textbf{algebra} $\mathfrak{F}$ is a vector space that is equipped with a
bilinear map $m: \mathfrak{F} \times \mathfrak{F} \rightarrow \mathfrak{F}$ termed
\emph{multiplication} and a nonzero element $1 \in \mathfrak{F}$ termed the \emph{unit} such
that $m(1,a) = m(a,1) = 1$. As usual, we abbreviate $m(a,b) = ab$ as $ab$. We do not assume
$\mathfrak{F}$ associative. Given an algebra, we freely think of real numbers as elements of
this algebra via the map $\omega \mapsto \omega 1$.

An algebra $\mathfrak{F}$ is a \textbf{division algebra} if given $a, b \in \mathfrak{F}$ with
$ab=0$, then either $a=0$ or $b=0$. Equivalently, $\mathfrak{F}$ is a division algebra if the
operation of left and right multiplications by any nonzero element is invertible. A
\textbf{normed division algebra} is an algebra $\mathfrak{F}$ that is also a normed vector
space with $||ab|| = ||a||||b||$. This implies that $\mathfrak{F}$ is a division algebra and
that $||1|| = 1$.

There are exactly four normed division algebras: real numbers ($\Re$), complex numbers
($\mathfrak{C}$), quaternions ($\mathfrak{H}$) and octonions ($\mathfrak{O}$), see
\citet{b:02}. We take into account that $\Re$, $\mathfrak{C}$, $\mathfrak{H}$ and
$\mathfrak{O}$ are the only normed division algebras; moreover, they are the only alternative
division algebras, and all division algebras have a real dimension of $1, 2, 4$ or $8$, which
is denoted by $\beta$, see \citet[Theorems 1, 2 and 3]{b:02}. In other branches of mathematics,
the parameter $\alpha = 2/\beta$ is used, see \citet{er:05}.

Let ${\mathcal L}^{\beta}_{m,n}$ be the linear space of all $n \times m$ matrices of rank $m
\leq n$ over $\mathfrak{F}$ with $m$ distinct positive singular values, where $\mathfrak{F}$
denotes a \emph{real finite-dimensional normed division algebra}. Let $\mathfrak{F}^{n \times
m}$ be the set of all $n \times m$ matrices over $\mathfrak{F}$. The dimension of
$\mathfrak{F}^{n \times m}$ over $\Re$ is $\beta mn$. Let $\mathbf{A} \in \mathfrak{F}^{n
\times m}$, then $\mathbf{A}^{*} = \overline{\mathbf{A}}^{T}$ denotes the usual conjugate
transpose.

The set of matrices $\mathbf{H}_{1} \in \mathfrak{F}^{n \times m}$ such that
$\mathbf{H}_{1}^{*}\mathbf{H}_{1} = \mathbf{I}_{m}$ is a manifold denoted ${\mathcal
V}_{m,n}^{\beta}$, is termed the \emph{Stiefel manifold} ($\mathbf{H}_{1}$ is also known as
\emph{semi-orthogonal} ($\beta = 1$), \emph{semi-unitary} ($\beta = 2$), \emph{semi-symplectic}
($\beta = 4$) and \emph{semi-exceptional type} ($\beta = 8$) matrices, see \citet{dm:99}). The
dimension of $\mathcal{V}_{m,n}^{\beta}$ over $\Re$ is $[\beta mn - m(m-1)\beta/2 -m]$. In
particular, ${\mathcal V}_{m,m}^{\beta}$ with dimension over $\Re$, $[m(m+1)\beta/2 - m]$, is
the maximal compact subgroup $\mathfrak{U}^{\beta}(m)$ of ${\mathcal L}^{\beta}_{m,m}$ and
consists of all matrices $\mathbf{H} \in \mathfrak{F}^{m \times m}$ such that
$\mathbf{H}^{*}\mathbf{H} = \mathbf{I}_{m}$. Therefore, $\mathfrak{U}^{\beta}(m)$ is the
\emph{real orthogonal group} $\mathcal{O}(m)$ ($\beta = 1$), the \emph{unitary group}
$\mathcal{U}(m)$ ($\beta = 2$), \emph{compact symplectic group} $\mathcal{S}p(m)$ ($\beta = 4$)
or \emph{exceptional type matrices} $\mathcal{O}o(m)$ ($\beta = 8$), for $\mathfrak{F} = \Re$,
$\mathfrak{C}$, $\mathfrak{H}$ or $\mathfrak{O}$, respectively.

We denote by ${\mathfrak S}_{m}^{\beta}$ the real vector space of all $\mathbf{S} \in
\mathfrak{F}^{m \times m}$ such that $\mathbf{S} = \mathbf{S}^{*}$. Let
$\mathfrak{P}_{m}^{\beta}$ be the \emph{cone of positive definite matrices} $\mathbf{S} \in
\mathfrak{F}^{m \times m}$; then $\mathfrak{P}_{m}^{\beta}$ is an open subset of ${\mathfrak
S}_{m}^{\beta}$. Over $\Re$, ${\mathfrak S}_{m}^{\beta}$ consist of \emph{symmetric} matrices;
over $\mathfrak{C}$, \emph{Hermitian} matrices; over $\mathfrak{H}$, \emph{quaternionic
Hermitian} matrices (also termed \emph{self-dual matrices}) and over $\mathfrak{O}$,
\emph{octonionic Hermitian} matrices. Generically, the elements of $\mathfrak{S}_{m}^{\beta}$
are termed
 \textbf{Hermitian matrices}, irrespective of the nature of $\mathfrak{F}$. The
dimension of $\mathfrak{S}_{m}^{\beta}$ over $\Re$ is $[m(m-1)\beta+2]/2$.

Let $\mathfrak{D}_{m}^{\beta}$ be the \emph{diagonal subgroup} of $\mathcal{L}_{m,m}^{\beta}$
consisting of all $\mathbf{D} \in \mathfrak{F}^{m \times m}$, $\mathbf{D} = \diag(d_{1},
\dots,d_{m})$.

For any matrix $\mathbf{X} \in \mathfrak{F}^{n \times m}$, $d\mathbf{X}$ denotes the\emph{
matrix of differentials} $(dx_{ij})$. Finally, we define the \emph{measure} or volume element
$(d\mathbf{X})$ when $\mathbf{X} \in \mathfrak{F}^{m \times n}, \mathfrak{S}_{m}^{\beta}$,
$\mathfrak{D}_{m}^{\beta}$ or $\mathcal{V}_{m,n}^{\beta}$, see \citet{d:02}.

If $\mathbf{X} \in \mathfrak{F}^{n \times m}$ then $(d\mathbf{X})$ (the Lebesgue measure in
$\mathfrak{F}^{n \times m}$) denotes the exterior product of the $\beta mn$ functionally
independent variables
$$
  (d\mathbf{X}) = \bigwedge_{i = 1}^{n}\bigwedge_{j = 1}^{m}dx_{ij} \quad \mbox{ where }
    \quad dx_{ij} = \bigwedge_{k = 1}^{\beta}dx_{ij}^{(k)}.
$$

If $\mathbf{S} \in \mathfrak{S}_{m}^{\beta}$ (or $\mathbf{S} \in \mathfrak{T}_{L}^{\beta}(m)$)
then $(d\mathbf{S})$ (the Lebesgue measure in $\mathfrak{S}_{m}^{\beta}$ or in
$\mathfrak{T}_{L}^{\beta}(m)$) denotes the exterior product of the $m(m+1)\beta/2$ functionally
independent variables (or denotes the exterior product of the $m(m-1)\beta/2 + n$ functionally
independent variables, if $s_{ii} \in \Re$ for all $i = 1, \dots, m$)
$$
  (d\mathbf{S}) = \left\{
                    \begin{array}{ll}
                      \displaystyle\bigwedge_{i \leq j}^{m}\bigwedge_{k = 1}^{\beta}ds_{ij}^{(k)}, &  \\
                      \displaystyle\bigwedge_{i=1}^{m} ds_{ii}\bigwedge_{i < j}^{m}\bigwedge_{k = 1}^{\beta}ds_{ij}^{(k)}, &
                       \hbox{if } s_{ii} \in \Re.
                    \end{array}
                  \right.
$$
The context generally establishes the conditions on the elements of $\mathbf{S}$, that is, if
$s_{ij} \in \Re$, $\in \mathfrak{C}$, $\in \mathfrak{H}$ or $ \in \mathfrak{O}$. It is
considered that
$$
  (d\mathbf{S}) = \bigwedge_{i \leq j}^{m}\bigwedge_{k = 1}^{\beta}ds_{ij}^{(k)}
   \equiv \bigwedge_{i=1}^{m} ds_{ii}\bigwedge_{i < j}^{m}\bigwedge_{k =
1}^{\beta}ds_{ij}^{(k)}.
$$
Observe, too, that for the Lebesgue measure $(d\mathbf{S})$ defined thus, it is required that
$\mathbf{S} \in \mathfrak{P}_{m}^{\beta}$, that is, $\mathbf{S}$ must be a non singular
Hermitian matrix (Hermitian definite positive matrix).

If $\mathbf{\Lambda} \in \mathfrak{D}_{m}^{\beta}$ then $(d\mathbf{\Lambda})$ (the Legesgue
measure in $\mathfrak{D}_{m}^{\beta}$) denotes the exterior product of the $\beta m$
functionally independent variables
$$
  (d\mathbf{\Lambda}) = \bigwedge_{i = 1}^{n}\bigwedge_{k = 1}^{\beta}d\lambda_{i}^{(k)}.
$$
If $\mathbf{H}_{1} \in \mathcal{V}_{m,n}^{\beta}$ then
$$
  (\mathbf{H}^{*}_{1}d\mathbf{H}_{1}) = \bigwedge_{i=1}^{n} \bigwedge_{j =i+1}^{m}
  \mathbf{h}_{j}^{*}d\mathbf{h}_{i}.
$$
where $\mathbf{H} = (\mathbf{H}_{1}|\mathbf{H}_{2}) = (\mathbf{h}_{1}, \dots,
\mathbf{h}_{m}|\mathbf{h}_{m+1}, \dots, \mathbf{h}_{n}) \in \mathfrak{U}^{\beta}(m)$. It can be
proved that this differential form does not depend on the choice of the $\mathbf{H}_{2}$
matrix. When $m = 1$; $\mathcal{V}^{\beta}_{1,n}$ defines the unit sphere in
$\mathfrak{F}^{n}$. This is, of course, an $(n-1)\beta$- dimensional surface in
$\mathfrak{F}^{n}$. When $m = n$ and denoting $\mathbf{H}_{1}$ by $\mathbf{H}$,
$(\mathbf{H}^{*}d\mathbf{H})$ is termed the \emph{Haar measure} on $\mathfrak{U}^{\beta}(m)$.

The surface area or volume of the Stiefel manifold $\mathcal{V}^{\beta}_{m,n}$ is
\begin{equation}\label{vol}
    \Vol(\mathcal{V}^{\beta}_{m,n}) = \int_{\mathbf{H}_{1} \in
  \mathcal{V}^{\beta}_{m,n}} (\mathbf{H}^{*}_{1}d\mathbf{H}_{1}) =
  \frac{2^{m}\pi^{mn\beta/2}}{\Gamma^{\beta}_{m}[n\beta/2]},
\end{equation}
where $\Gamma^{\beta}_{m}[a]$ denotes the multivariate Gamma function for the space
$\mathfrak{S}_{m}^{\beta}$, and is defined by
\begin{eqnarray*}
  \Gamma_{m}^{\beta}[a] &=& \displaystyle\int_{\mathbf{A} \in \mathfrak{P}_{m}^{\beta}}
  \etr\{-\mathbf{A}\} |\mathbf{A}|^{a-(m-1)\beta/2 - 1}(d\mathbf{A}) \\
    &=& \pi^{m(m-1)\beta/4}\displaystyle\prod_{i=1}^{m} \Gamma[a-(i-1)\beta/2],
\end{eqnarray*}
where $\etr(\cdot) = \exp(\tr(\cdot))$, $|\cdot|$ denotes the determinant and $\re(a)
> (m-1)\beta/2$, see \citet{gr:87}.

\subsection{Jacobians and and some results on integration}\label{sec22}

First, we summarise diverse Jacobians in terms of the $\beta$ parameter, some based on the work
of \citet{d:02}, while other results are proposed as extensions of real, complex or quaternion
cases, see  \citet{j:54}, \citet{j:64}, \citet{k:65}, \citet{me:91}, \citet{rva:05a} and
\citet{lx:09}.

\begin{lem}[Singular value decomposition, $SVD$]\label{lemsvd}
Let $\mathbf{X} \in {\mathcal L}_{m,n}^{\beta}$, such that $\mathbf{X} =
\mathbf{V}_{1}\mathbf{DW}^{*}$ with $\mathbf{V}_{1} \in {\mathcal V}_{m,n}^{\beta}$,
$\mathbf{W} \in \mathfrak{U}^{\beta}(m)$ and $\mathbf{D} = \diag(d_{1}, \cdots,d_{m}) \in
\mathfrak{D}_{m}^{1}$, $d_{1}> \cdots > d_{m} > 0$. Then
\begin{equation}\label{svd}
    (d\mathbf{X}) = 2^{-m}\pi^{\tau} \prod_{i = 1}^{m} d_{i}^{\beta(n - m + 1) -1}
    \prod_{i < j}^{m}(d_{i}^{2} - d_{j}^{2})^{\beta} (d\mathbf{D}) (\mathbf{V}_{1}^{*}d\mathbf{V}_{1})
    (\mathbf{W}^{*}d\mathbf{W}),
\end{equation}
where
$$
  \tau = \left\{
             \begin{array}{rl}
               0, & \beta = 1; \\
               -m, & \beta = 2; \\
               -2m, & \beta = 4; \\
               -4m, & \beta = 8.
             \end{array}
           \right.
$$
\end{lem}

\begin{lem}[ Spectral decomposition]\label{lemsd}
Let $\mathbf{S} \in \mathfrak{P}_{m}^{\beta}$. Then the spectral decomposition can be written
as $\mathbf{S} = \mathbf{W}\mathbf{\Lambda W}^{*}$, where $\mathbf{W} \in
\mathfrak{U}^{\beta}(m)$ and $\mathbf{\Lambda} = \diag(\lambda_{1}, \dots, \lambda_{m}) \in
\mathfrak{D}_{m}^{1}$, with $\lambda_{1}> \cdots> \lambda_{m}>0$. Then
\begin{equation}\label{sd}
    (d\mathbf{S}) = 2^{-m} \pi^{\tau} \prod_{i < j}^{m} (\lambda_{i} - \lambda_{j})^{\beta}
    (d\mathbf{\Lambda})(\mathbf{W}^{*}d\mathbf{W}),
\end{equation}
where $\tau$ is defined in Lemma \ref{lemsvd}.
\end{lem}

\begin{lem}\label{lemW}
Let $\mathbf{X} \in {\mathcal L}_{m,n}^{\beta}$, and  $\mathbf{S} = \mathbf{X}^{*}\mathbf{X}
\in \mathfrak{P}_{m}^{\beta}.$ Then
\begin{equation}\label{w}
    (d\mathbf{X}) = 2^{-m} |\mathbf{S}|^{\beta(n - m + 1)/2 - 1}
    (d\mathbf{S})(\mathbf{V}_{1}^{*}d\mathbf{V}_{1}).
\end{equation}
\end{lem}

\begin{thm}\label{theoi}
Let $\mathbf{S} \in \mathfrak{P}_{m}^{\beta}.$ Then ignoring the sign, if $\mathbf{Y} =
\mathbf{S}^{-1}$
\begin{equation}\label{i}
    (d\mathbf{Y}) = |\mathbf{S}|^{-\beta(m - 1) - 2}(d\mathbf{S}).
\end{equation}
\end{thm}
Consider the following property of Jack polynomials.
\begin{lem}\label{lemj}
If $\mathbf{X} \in \mathfrak{L}^{\beta}_{m,n}$, then
\begin{equation}\label{eqj}
    \int_{\mathbf{H}_{1} \in \mathcal{V}_{m,n}^{\beta}} (\tr(\mathbf{XH}_{1}))^{2k} (d\mathbf{H}_{1}) =
  \sum_{\kappa}\frac{\left(
  \frac{1}{2}\right)_{k}}{[\beta n/2]^{\beta}_{\kappa}}C^{\beta}_{\kappa}(\mathbf{XX}^{*}),
\end{equation}
where $C_{\kappa}^{\beta}(\mathbf{B})$ are the Jack polynomials of weight $\kappa$ of
$\mathbf{B} \in {\mathfrak S}_{m}^{\beta}$ corresponding to the partition $\kappa=(k_{1},\ldots
k_{m})$ of $k$, $k_{1} \geq \cdots \geq k_{m} \geq 0$ with $\sum_{i=1}^{m}k_{i}=k$, see
\citet{S:97}, \citet{gr:87}; and $[a]^{\beta}_{\kappa}$ denotes the generalised Pochhammer
symbol of weight $\kappa$, defined as
$$
  [a]^{\beta}_{\kappa} = \prod_{i=1}^{m}(a-(i-1)\beta/2)_{k_{1}},
$$
where $ \Re(a) > (m-1)\beta/2 - k_{m}$ and $(a)_{i} = a(a+1)\cdots(a+i-1)$.
\end{lem}
\begin{proof}
See \citet{dggj:10}.\qed
\end{proof}
Now, we utilise the complexification $\mathfrak{S}_{m}^{\beta, \mathfrak{C}} =
\mathfrak{S}_{m}^{\beta} + i \mathfrak{S}_{m}^{\beta}$ of $\mathfrak{S}_{m}^{\beta}$. That is,
$\mathfrak{S}_{m}^{\beta, \mathfrak{C}}$ consist of all matrices $\mathbf{X} \in
(\mathfrak{F^{\mathfrak{C}}})^{m \times m}$ of the form $\mathbf{Z} = \mathbf{X} +
i\mathbf{Y}$, with $\mathbf{X}, \mathbf{Y} \in \mathfrak{S}_{m}^{\beta}$. We refer to
$\mathbf{X} = \re(\mathbf{Z})$ and $\mathbf{Y} = \im(\mathbf{Z})$ as the \emph{real and
imaginary parts} of $\mathbf{Z}$, respectively. The \emph{generalised right half-plane}
$\mathbf{\Phi} = \mathfrak{P}_{m}^{\beta} + i \mathfrak{S}_{m}^{\beta}$ in
$\mathfrak{S}_{m}^{\beta,\mathfrak{C}}$ consists of all $\mathbf{Z} \in
\mathfrak{S}_{m}^{\beta,\mathfrak{C}}$ such that $\re(\mathbf{Z}) \in
\mathfrak{P}_{m}^{\beta}$, see \citet[p. 801]{gr:87}. Also, as in \citet{da:80}, consider the
following notation,
$$
  \sum_{k,l = 0}^{\infty} \sum_{\kappa,\delta;\phi \in \kappa\cdot\delta} \equiv
  \sum_{k = 0}^{\infty} \sum_{l = 0}^{\infty} \sum_{\kappa}\sum_{\delta}\sum_{\phi \in
  \kappa\cdot\delta}.
$$
$C_{\phi}^{[\beta]\kappa, \delta}(\mathbf{A}, \mathbf{B})$ denotes the invariant polynomials,
which are defined in \citet{da:79} and \citet{da:80} in the real case. \citet{dg:09} studied
these invariant polynomials and many of their basic properties for real normed division
algebras.
\begin{thm}\label{teohf}
Let $\mathbf{\Delta} \in \mathbf{\Phi}$ then
$$
  \int_{\mathbf{0} < \mathbf{X} < \mathbf{\Delta}} |\mathbf{X}|^{a-(m-1)\beta/2-1}
  \etr\{\mathbf{-XA}\} {}_{p}F_{q}^{\beta}(a_{1},\dots,a_{p};b_{1},\dots,b_{q};\mathbf{BX})
  (d\mathbf{X})
$$
\begin{eqnarray*}
  && = \frac{\Gamma_{m}^{\beta}[a]\Gamma_{m}^{\beta}[(m-1)\beta/2+1]}{\Gamma_{m}^{\beta}[a+(m-1)\beta/2+1]}
  |\mathbf{\Delta}|^{a} \\
  && \qquad\times \sum_{k,l=0}^{\infty}\sum_{\phi \in \kappa\cdot\delta}
  \frac{[a_{1}]_{\kappa}^{\beta}, \dots, [a_{p}]_{\kappa}^{\beta}}
  {k! l![b_{1}]_{\kappa}^{\beta}, \dots, [a_{q}]_{\kappa}^{\beta}}
  \frac{\theta_{\phi}^{[\beta]\kappa,\delta} C_{\phi}^{[\beta]\kappa,\delta}
  (-\mathbf{A\Delta}, \mathbf{B\Delta})}{[a+(m-1)\beta/2+1]_{\phi}^{\beta}}
\end{eqnarray*}
where $\theta_{\phi}^{[\beta]\kappa,\delta}$ is defined in \citet[eq. (52)]{dg:09}, see also
\citet{da:80}. Also, ${}_{q}F_{p}$ denotes the hypergeometric function defined in terms of Jack
polynomials, see \citet{gr:87} and \citet{KE:06}.
\end{thm}
\begin{proof}
This follows immediately, expanding ${}_{p}F_{q}^{\beta}$ in terms of Jack polynomials and
using \citet[eq. (5.38)]{dg:09}.
\end{proof}

\section{Wishart distribution}\label{sec3}

Recall that $\mathbf{X}\in \mathfrak{L}^{\beta}_{m,n}$ has a matrix multivariate elliptically
contoured distribution for real normed division algebras if its density, with respect to the
Lebesgue measure, is given by (see \citet{jdggj:09}):
$$
  f_{\mathbf{X}}(\mathbf{X})=\frac{C^{\beta}(m,n)}{|\mathbf{\Sigma}|^{\beta n/2}|\mathbf{\Theta}|^{\beta m/2}}
  h\left\{\tr\left[\mathbf{\Sigma}^{-1}(\mathbf{X}-\boldsymbol{\mu})^{*}\mathbf{\Theta}^{-1}
  (\mathbf{X}- \boldsymbol{\mu})\right]\right\},
$$
where  $\boldsymbol{\mu}\in \mathfrak{L}^{\beta}_{m,n}$, $ \mathbf{\Sigma}\in
\mathfrak{P}^{\beta}_{m}$,  $ \mathbf{\Theta}\in \mathfrak{P}^{\beta}_{m}$. The function $h:
\mathfrak{F} \rightarrow [0,\infty)$ is termed the generator function, and it is such that
$\int_{\mathfrak{P}^{\beta}_{1}} u^{\beta nm-1}h(u^2)du < \infty$ and
$$
  C^{\beta}(m,n) = \frac{\Gamma[\beta mn/2]}{2 \pi^{\beta mn/2}} \left\{
    \int_{\mathfrak{P}^{\beta}_{1}} u^{\beta nm-1}h(u^2)du\right \}
$$
Such a distribution is denoted by $\mathbf{X}\sim \mathcal{E}^{\beta}_{n\times
m}(\boldsymbol{\mu},\mathbf{\Sigma}, \mathbf{\Theta}, h)$, for the real case see \citet{fz:90}
and \citet{gv:93} and \citet{mdm:06} for the complex case. Observe that this class of matrix
multivariate distributions includes normal, contaminated normal, Pearson type II and VI, Kotz,
Jensen-Logistic, power exponential and Bessel distributions, among others; these distributions
have tails that are more or less weighted, and/or present a greater or smaller degree of
kurtosis than the normal distribution.

\begin{thm}\label{theogw}
Let $\mathbf{S} = \mathbf{X}^{*}\mathbf{\Theta}^{-1}\mathbf{X} \in \mathfrak{P}_{m}^{\beta}$.
$\mathbf{S}$ is said to have a generalised Wishart distribution for a real normed division
algebra, this fact being denoted as $\mathbf{S} \sim \mathcal{GW}_{m}^{\beta}(n,
\mathbf{\Sigma}, \mathbf{\Omega}, h)$. Moreover, its density function is
\begin{equation}\label{gweq}
    \frac{\pi^{\beta mn/2} C^{\beta}(m,n) }{\Gamma_{m}^{\beta}[\beta n/2]
    |\mathbf{\Sigma}|^{\beta n/2}} |\mathbf{S}|^{\beta(n-m+1)/2-1}\sum_{k=0}^{\infty}
    \frac{h^{(2k)}\left(\tr \mathbf{\Sigma}^{-1}\mathbf{S} + \mathbf{\Omega}\right)}{k!}
    \frac{C_{\kappa}^{\beta}\left(\mathbf{\Omega \Sigma}^{-1}\mathbf{S}\right)}
    {[\beta n/2]_{\kappa}^{\beta}}
\end{equation}
where $\mathbf{\Omega} = \mathbf{\Sigma}^{-1}\pmb{\mu}^{*}\mathbf{\Theta}^{-1}\pmb{\mu}$ and
$h^{(j)}(\cdot)$ is the $j$th derivative of $h$ with respect to $v=\tr \mathbf{\Sigma^{-1}
\mathbf{S}}$.
\end{thm}
\begin{proof}
Let $\mathbf{S} = \mathbf{X}^{*}\mathbf{\Theta}^{-1} \mathbf{X} = \mathbf{Y}^{*}\mathbf{Y}$,
where
$$
  \mathbf{Y} = \mathbf{\Theta}^{-1/2} \mathbf{X} \sim
  \mathcal{E}_{n \times m}^{\beta}(\mathbf{\Theta}^{-1/2}\pmb{\mu}, \mathbf{\Sigma},
  \mathbf{I}_{m}, h),
$$
with $(\mathbf{\Theta}^{1/2})^{2} = \mathbf{\Theta}$, and so
$$
  f_{\mathbf{Y}}(\mathbf{Y}) = \frac{C^{\beta}(m,n)}{|\mathbf{\Sigma}|^{\beta n/2}} h\left[\tr \mathbf{\Sigma}^{-1}
   (\mathbf{Y} - \pmb{\mu})^{*} (\mathbf{Y} - \pmb{\mu})\right]
$$
Let us now consider the singular value decomposition of matrix $\mathbf{Y} =
\mathbf{V}_{1}\mathbf{DW}^{*}$. Then by Lemma \ref{lemsvd}, the joint density function of
$\mathbf{V}_{1}$, $\mathbf{D}$ and $\mathbf{W}$ is
\begin{eqnarray}
   && \hspace{-1cm}\frac{2^{-m} C^{\beta}(m,n)\pi^{\tau} \displaystyle\prod_{i = 1}^{m} d_{i}^{\beta(n - m + 1) -1}
    \prod_{i < j}^{m}(d_{i}^{2} - d_{j}^{2})^{\beta}}{|\mathbf{\Sigma}|^{\beta
   n/2}}\nonumber\\
   & & \times \ h\left[\tr \left(\mathbf{\Sigma}^{-1}\mathbf{WD}^{2}\mathbf{W}^{*} + \mathbf{\Omega}\right) + \tr
   \left(-2 \boldsymbol{\mu}_{\mathbf{Y}} \mathbf{\Sigma}^{-1} \mathbf{WDV}_{1}^{*}\right)
   \right]\nonumber\\  \label{eq1}
   &&  \times \ (\mathbf{W}^{*}d\mathbf{W})(d\mathbf{D})(\mathbf{V}_{1}^{*}d\mathbf{V}_{1}),
\end{eqnarray}
where $\pmb{\mu}_{\mathbf{Y}} = \mathbf{\Theta}^{-1/2}\pmb{\mu}$ and
$\mathbf{\Sigma}^{-1}\pmb{\mu}^{*}\mathbf{\Theta}^{-1}\pmb{\mu}$. Let us now assume that $h$
can be expanded in series of power, that is
$$
  h(v+a) = \sum_{k=0}^{\infty} \frac{h^{(k)}(a) v^{k}}{k!}.
$$
Hence, considering only $h$ in (\ref{eq1})
$$
  = \sum_{k=0}^{\infty}\frac{h^{(k)}\left[\tr \left(\mathbf{\Sigma}^{-1}\mathbf{WD}^{2}\mathbf{W}^{*} +
  \mathbf{\Omega}\right)\right]}{k!} \left(\tr\left( -2 \boldsymbol{\mu}_{\mathbf{Y}} \mathbf{\Sigma}^{-1}
  \mathbf{WDV}_{1}^{*}\right)\right)^{k}.
$$
And from Lemma \ref{lemj} noting that (\ref{eqj}) is zero for all odd $k$,\\[2ex]%
$
    \displaystyle\int_{\mathbf{H} \in\mathcal{V}_{m,n}^{\beta}}\left[\tr\left(-2\tr\boldsymbol{\mu}_{\mathbf{Y}}^{*}
    \boldsymbol{\Sigma}^{-1}\mathbf{WDV}_{1}^{*}\right)\right]^{2k}
    (\mathbf{V}_{1}^{*}d\mathbf{V}_{1})
$
\par \noindent \hfill
\hbox{$=\displaystyle\frac{2^{m}\pi^{\beta mn/2}}{\Gamma_{m}\left[
    \beta n/2\right]} \sum_{\kappa} \frac{\left(\frac{1}{2}\right)_{k}
    4^{k}}{[\beta n/2]_{\kappa}^{\beta}}
    C_{\kappa}\left(\boldsymbol{\Omega}\boldsymbol{\Sigma}^{-1}\mathbf{W
    D}^{2}\mathbf{W}^{*}\right).$}
\par \noindent\newline
Observing that $\left.\left(\frac{1}{2}\right)_{k}4^{k}\right/(2k)!=1/k!$, the joint density
function of $\mathbf{D}$ and $\mathbf{W}$ is
\begin{eqnarray*}
   && \hspace{-1cm}\frac{\pi^{\beta mn/2+\tau}C^{\beta}(m,n) \displaystyle\prod_{i = 1}^{m} d_{i}^{\beta(n - m + 1) -1}
    \prod_{i < j}^{m}(d_{i}^{2} - d_{j}^{2})^{\beta}}{\Gamma_{m}\left[
    \beta n/2\right] |\mathbf{\Sigma}|^{\beta n/2}}\nonumber\\
   & & \times \ \sum_{k=0}^{\infty}\frac{h^{(2k)}\left[\tr \left(\mathbf{\Sigma}^{-1}\mathbf{WD}^{2}\mathbf{W}^{*} +
    \mathbf{\Omega}\right)\right]}{k!} \sum_{\kappa} \frac{C_{\kappa}\left(\boldsymbol{\Omega}\boldsymbol{\Sigma}^{-1}\mathbf{W
    D}^{2}\mathbf{W}^{*}\right)}{[\beta n/2]_{\kappa}^{\beta}}
    \nonumber\\
    &&  \times \ (\mathbf{W}^{*}d\mathbf{W})(d\mathbf{D}).
\end{eqnarray*}
Finally, let $\mathbf{S} = \mathbf{Y}^{*}\mathbf{Y} = \mathbf{WD}^{2}\mathbf{W}^{*}$. The
desired result is obtained from lemmas \ref{lemsd} and \ref{lemW}, noting that $(d\mathbf{D}) =
2^{-m}|\mathbf{D}^{2}|^{-1/2}(d\mathbf{D}^{2})$ and $ \prod_{i=1}^{m}d_{i}^{2} = |\mathbf{S}|$.
\qed
\end{proof}

Distribution (\ref{gweq}) was found by \citet{dggj:06} for the real case and for the general
central case by \citet{jdggj:09}.

\begin{cor} Assume that $\mathbf{X}$ is a matrix multivariate normal distribution for real normed
division algebras. Then $\mathbf{S} = \mathbf{X}^{*}\mathbf{\Theta}^{-1}\mathbf{X}$ has a
Wishart distribution for real normed division algebras and is denoted as $\mathbf{S} \sim
\mathcal{W}_{m}^{\beta}(n,\mathbf{\Sigma},\mathbf{\Omega})$. Moreover its density is
\begin{eqnarray}
  && \frac{1}{(2/\beta)^{\beta mn/2}\Gamma_{m}^{\beta}[\beta n/2]
    |\mathbf{\Sigma}|^{\beta n/2}} |\mathbf{S}|^{\beta(n-m+1)/2-1}
    \etr \{-\beta\left(\mathbf{\Sigma}^{-1}\mathbf{S} + \mathbf{\Omega}\right)/2\} \nonumber\\
    \label{gweqN}
  && \hspace{6.5cm}\times \ {}_{0}F_{1}^{\beta}\left(\beta n/2; \beta^{2}\mathbf{\Omega
  \Sigma}^{-1}\mathbf{S}/4\right).
\end{eqnarray}
\end{cor}
\begin{proof}
This follows from (\ref{gweq}) taking into account that for the normal case, $h(u) =
\exp\{-\beta u/2\}$ and $C^{\beta}(m,n) = (2 \pi/\beta)^{-\beta mn/2}$. \qed
\end{proof}

This result has been found by \citet{h:55} and \citet{j:61} for the real case; by \citet{j:64},
\citet{k:65} and \citet{rva:05a} for the complex case and by \citet{lx:09} for the central
quaternion case, among other authors. The general central case of (\ref{gweqN}) was found by
\citet{jdggj:09}.

From Theorems \ref{theoi} and \ref{theogw} it is straightforward to obtain the distribution of
$\mathbf{S}^{-1}$, termed the inverse generalised Wishart distribution.

\begin{cor}
In Theorem \ref{theogw} we define $\mathbf{W} = \mathbf{S}^{-1}$. Then its density function is
\begin{equation}\label{igweq}
    \frac{\pi^{\beta mn/2} C^{\beta}(m,n)|\mathbf{W}|^{\beta(n+m+1)/2-3}}{\Gamma_{m}^{\beta}[\beta n/2]
    |\mathbf{\Sigma}|^{\beta n/2}} \sum_{k=0}^{\infty}
    \frac{h^{(2k)}\left(\tr \mathbf{\Sigma}^{-1}\mathbf{W}^{-1} + \mathbf{\Omega}\right)}{k!}
    \frac{C_{\kappa}^{\beta}\left(\mathbf{\Omega \Sigma}^{-1}\mathbf{W}^{-1}\right)}
    {[\beta n/2]_{\kappa}^{\beta}}
\end{equation}
\end{cor}
The density (\ref{igweq}) was found by \citet{ihl:07} in the real case.

\section{Eigenvalue densities}\label{sec4}

In this section we find the general joint density function of the eigenvalues of $\mathbf{S}$
and the density of $\lambda_{max}$ for the normal case.

\begin{thm}\label{jeteo}
Assume that $\mathbf{S} \sim \mathcal{GW}_{m}^{\beta}(n, \mathbf{\Sigma}, \mathbf{\Omega}, h)$.
Then the joint density of eigenvalues $\lambda_{1}, \dots, \lambda_{m} > 0$, of $\mathbf{S}$ is
\begin{eqnarray}
   && \hspace{-1cm} \frac{\pi^{\beta(mn + m^{2})/2 + \tau} C^{\beta}(m,n)
   \displaystyle\prod_{i = 1}^{m} \lambda_{i}^{\beta(n-m+1)/2 -1}
   \prod_{i < j}^{m}(\lambda_{i} - \lambda_{j})^{\beta}}{\Gamma_{m}^{\beta}[\beta n/2] \Gamma_{m}^{\beta}[\beta m/2]
   |\mathbf{\Sigma}|^{\beta n/2}} \nonumber\\\label{jde}
   && \quad \times \ \sum_{k,l = 0}^{\infty} \sum_{\kappa,\delta;\phi \in \kappa\cdot\delta} \frac{h^{(2k+l)}
   (\tr \mathbf{\Omega}) \theta_{\phi}^{[\beta]\kappa, \delta}}{k! l! [\beta n/2]_{\kappa}^{\beta}}
   \frac{C_{\phi}^{\beta}(\mathbf{\Lambda}) C_{\phi}^{[\beta]\kappa, \delta}(\mathbf{\Sigma}^{-1},
   \mathbf{\Omega\Sigma}^{-1})}{C_{\phi}^{\beta}(\mathbf{I})}.
\end{eqnarray}
\end{thm}
\begin{proof}
Let $\mathbf{S} = \mathbf{W \Lambda W}^{*}$ the spectral decomposition of $\mathbf{S}$, where
$\mathbf{W} \in \mathfrak{U}^{\beta}(m)$ and $\mathbf{\Lambda} = \diag(\lambda_{1}, \dots,
\lambda_{m}) \in \mathfrak{D}^{1}(m)$, $\lambda_{1}> \cdots> \lambda_{m} > 0$. Then by
(\ref{gweq}) and Lemma \ref{lemsd}, the marginal density function of $\mathbf{\Lambda}$ is
\begin{eqnarray}
   && \hspace{-1cm}\frac{2^{-m}\pi^{\beta mn/2+\tau} C^{\beta}(m,n)
   \displaystyle\prod_{i = 1}^{m} \lambda_{i}^{\beta(n-m+1)/2 -1}
   \prod_{i < j}^{m}(\lambda_{i} - \lambda_{j})^{\beta}}{\Gamma_{m}^{\beta}[\beta n/2]
   |\mathbf{\Sigma}|^{\beta n/2}} \sum_{k=0}^{\infty}\frac{1}{k! [\beta n/2]_{\kappa}^{\beta}}\nonumber\\
   && \label{jed} \hspace{-1cm}\times \
   \int_{\mathbf{W} \in \mathfrak{U}^{\beta}(m)}h^{(2k)}\left(\tr \mathbf{\Sigma}^{-1}
   \mathbf{W \Lambda W}^{*} + \mathbf{\Omega}\right)
   C_{\kappa}^{\beta}\left(\mathbf{\Omega \Sigma}^{-1}\mathbf{W \Lambda W}^{*}\right)
   (\mathbf{W}^{*}d\mathbf{W}).
\end{eqnarray}
By denoting the integral in (\ref{jed}) by $J$ and the expanding $h^{(2k)}$ into series of
powers, we have
\begin{eqnarray}
  J &=& \sum_{l=0}^{\infty}\frac{h^{(2k + l)}\left(\tr \mathbf{\Omega}\right)}{l!}\nonumber\\
    && \ \times \ \int_{\mathbf{W} \in \mathfrak{U}^{\beta}(m)}\left(\tr \mathbf{\Sigma}^{-1}
    \mathbf{W \Lambda W}^{*}\right)^{l} C_{\kappa}^{\beta}\left(\mathbf{\Omega
    \Sigma}^{-1}\mathbf{W \Lambda W}^{*}\right)
    (\mathbf{W}^{*}d\mathbf{W}) \nonumber\\
    &=& \sum_{l=0}^{\infty}\sum_{\delta}\frac{h^{(2k + l)}\left(\tr \mathbf{\Omega}\right)}{l!}\nonumber\\
    && \ \times \
    \int_{\mathbf{W} \in \mathfrak{U}^{\beta}(m)}C_{\delta}^{\beta}\left(\mathbf{\Sigma}^{-1}
    \mathbf{W \Lambda W}^{*}\right) C_{\kappa}^{\beta}\left(\mathbf{\Omega
    \Sigma}^{-1}\mathbf{W \Lambda W}^{*}\right)
    (\mathbf{W}^{*}d\mathbf{W})\nonumber\\
    &=& \frac{2^{m}\pi^{\beta m^{2}/2}}{\Gamma_{m}^{\beta}[\beta m/2]}\nonumber\\
    && \label{ijed}\ \times \
    \sum_{l=0}^{\infty}\sum_{\delta;\phi \in \kappa\cdot\delta}\frac{h^{(2k + l)}\left(\tr
    \mathbf{\Omega}\right) \theta_{\phi}^{[\beta]\kappa,\delta}}{l!}
    \frac{C_{\phi}^{\beta}\left(\mathbf{\Lambda}\right) C_{\phi}^{[\beta]\kappa,\delta}
    \left(\mathbf{\Sigma^{-1},\Omega
    \Sigma}^{-1}\right)}{C_{\phi}^{\beta}(\mathbf{I})}.
\end{eqnarray}
The last equality is obtained by applying \citet[eq. (5.1)]{dg:09}. The desired result then
follows by substituting (\ref{ijed}) in (\ref{jed}). \qed
\end{proof}
In the real case (\ref{jde}) was obtained by \citet{dggj:06}.

\begin{cor}\label{coro3}
Assume that $\mathbf{S} \sim \mathcal{W}_{m}^{\beta}(n, \mathbf{\Sigma}, \mathbf{\Omega})$.
Then the joint density of eigenvalues $\lambda_{1}, \dots, \lambda_{m} > 0$, of $\mathbf{S}$ is
\begin{eqnarray*}
   && \hspace{-1cm}\frac{\pi^{\beta m^{2}/2 + \tau}  \displaystyle\prod_{i = 1}^{m} \lambda_{i}^{\beta(n-m+1)/2 -1}
   \prod_{i < j}^{m}(\lambda_{i} - \lambda_{j})^{\beta} \etr\{-\beta \mathbf{\Omega}/2\}}
   {(2/\beta)^{\beta mn/2} \Gamma_{m}^{\beta}[\beta n/2] \Gamma_{m}^{\beta}[\beta m/2]
   |\mathbf{\Sigma}|^{\beta n/2}} \\
   && \quad \times \ \sum_{k,l = 0}^{\infty} \sum_{\kappa,\delta;\phi \in \kappa\cdot\delta}
   \frac{\theta_{\phi}^{[\beta]\kappa, \delta}}{k! l! [\beta n/2]_{\kappa}^{\beta}}
   \frac{C_{\phi}^{\beta}(\mathbf{\Lambda}) C_{\phi}^{[\beta]\kappa, \delta}(\mathbf{- \beta\Sigma}^{-1}/2,
   \beta^{2} \mathbf{\Omega\Sigma}^{-1}/4)}{C_{\phi}^{\beta}(\mathbf{I})}
\end{eqnarray*}
\end{cor}
\begin{proof}
The proof follows from (\ref{jde}) taking into account that for the normal case, $h(u) =
\exp\{-\beta u/2\}$ and $C^{\beta}(m,n) = (2 \pi/\beta)^{-\beta mn/2}$ from where
$h^{(2k+l)}(u) = (-\beta/2)^{2k+l}\exp\{-\beta u/2\}$ and observing that
$C_{\phi}^{[\beta]\kappa, \delta}(a\mathbf{A}, b\mathbf{B}) = a^{k} b^{l}
C_{\phi}^{[\beta]\kappa, \delta}(\mathbf{A}, \mathbf{B})$, see \citet[eq. (5.8)]{dg:09}. \qed
\end{proof}
The result in Corollary \ref{coro3} was obtained by \citet{da:80} for the real case; by
\citet{rva:05a} for the complex case, and by \citet{lx:09} for the central quaternion case.

\begin{thm}\label{teoDF}
Let $\mathbf{\Delta} \in \mathbf{\Phi}$ and consider that $\mathbf{S} \sim
\mathcal{W}_{m}^{\beta}(n,\mathbf{\Sigma},\mathbf{\Omega})$ then the probability $\P[\mathbf{S}
< \mathbf{\Delta}]$ is
\begin{eqnarray*}
  && \hspace{-1cm}\frac{\Gamma_{m}^{\beta}[(m-1)\beta/2+1] |\mathbf{\Delta}|^{\beta n/2}
    \etr \{-\beta \mathbf{\Omega}/2\}}
    {(2/\beta)^{\beta mn/2}\Gamma_{m}^{\beta}[\beta (n+m-1)/2+1]
    |\mathbf{\Sigma}|^{\beta n/2}}  \nonumber\\
  && \quad \times \ \sum_{k,l = 0}^{\infty} \sum_{\kappa,\delta;\phi \in \kappa\cdot\delta}
   \frac{[\beta n/2]_{\phi}^{\beta}}{k! l! [\beta n/2]_{\kappa}^{\beta}}
   \frac{\theta_{\phi}^{[\beta]\kappa, \delta} C_{\phi}^{[\beta]\kappa, \delta}
   (\mathbf{- \beta\Sigma}^{-1}\mathbf{\Delta}/2,
   \beta^{2}\mathbf{\Omega\Sigma}^{-1}\mathbf{\Delta}/4)}
   {[\beta(n+m-1)/2+1]_{\phi}^{\beta}}.
\end{eqnarray*}
\end{thm}
\begin{proof}
The proof follows from (\ref{gweqN}) and Theorem \ref{teohf}. \qed
\end{proof}

Note that $\lambda_{\max} < y$ is equivalent to $\mathbf{S} < y\mathbf{I}$. Therefore, the
distribution of $\lambda_{\max}$ is obtained by letting $\mathbf{\Delta} = y \mathbf{I}$ in
Theorem \ref{teoDF}, and hence:

\begin{cor}
Assume that $\mathbf{S} \sim \mathcal{W}_{m}^{\beta}(n,\mathbf{\Sigma},\mathbf{\Omega})$  and
let $\lambda_{\max}$ be the maximum eigenvalue of $\mathbf{S}$. Then $\P[\lambda_{\max} < y]$
is
\begin{eqnarray*}
  && \hspace{-1cm}\frac{\Gamma_{m}^{\beta}[(m-1)\beta/2+1] y^{\beta mn/2}
    \etr \{-\beta \mathbf{\Omega}/2\}}
    {(2/\beta)^{\beta mn/2}\Gamma_{m}^{\beta}[\beta (n+m-1)/2+1]
    |\mathbf{\Sigma}|^{\beta n/2}}  \nonumber\\
  && \quad \times \ \sum_{k,l = 0}^{\infty} \sum_{\kappa,\delta;\phi \in \kappa\cdot\delta}
   \frac{[\beta n/2]_{\phi}^{\beta}}{k! l! [\beta n/2]_{\kappa}^{\beta}}
   \frac{\theta_{\phi}^{[\beta]\kappa, \delta} y^{k+l} C_{\phi}^{[\beta]\kappa,
   \delta}(\mathbf{- \beta\Sigma}^{-1}/2,
   \beta^{2} \mathbf{\Omega\Sigma}^{-1}/4)}{[\beta(n+m-1)/2+1]_{\phi}^{\beta}}.
\end{eqnarray*}
\end{cor}
This latter result was found by \citet{rva:05a} for the complex case.

\section*{Conclusions}

As shown in this paper, it is possible to take a general approach to the theory of
distributions, for the real, complex quaternion and octonion cases simultaneously. However, any
generalisation entails a cost, in this case the need to use some concepts, definitions and
notation from abstract algebra. Thus, any reader interested in a particular case -- real,
complex, quaternions or octonions -- need simply take the particular value of $\beta$ in order
to obtain the results desired.

In summary, the noncentral generalised Wishart distribution, the joint distribution of the
eigenvalues and the maximum eigenvalue distribution are found under a unified approach that
allows the simultaneous study of the real, complex, quaternion and octonion cases, generically
termed distributions for real normed division algebras.

\section*{Acknowledgements}
This research work was partially supported  by CONACYT-M\'exico, Research Grant No. \ 81512 and
IDI-Spain, Grants No. FQM2006-2271 and MTM2008-05785. This paper was written during J. A.
D\'{\i}az-Garc\'{\i}a's stay as a visiting professor at the Department of Statistics and O. R.
of the University of Granada, Spain.

\end{document}